\documentclass[centering,11pt,reqno]{amsart}  

\usepackage[utf8]{inputenc}
\usepackage[T1]{fontenc}
\usepackage{amsmath,amsthm}
\usepackage{amsfonts,amssymb}
\usepackage[mathscr]{eucal}
\usepackage{url}
\usepackage{paralist}
\usepackage[colorlinks,cite color=blue,pagebackref=true,pdftex]{hyperref}
\usepackage[margin=2.5cm]{geometry}
\usepackage{tikz}
\usetikzlibrary{calc,shapes}
\usepackage{mathtools}
\usepackage{blkarray}
\usepackage[capitalize]{cleveref}
 \usepackage[foot]{amsaddr}
\usepackage{tikz-cd}

\newtheorem{theorem}{Theorem}[section]
\newtheorem*{lemma*}{Lemma}
\newtheorem*{prop*}{Proposition}
\newtheorem{lem}[theorem]{Lemma}
\newtheorem{cor}[theorem]{Corollary}

\theoremstyle{definition}

\newcommand{\vspan}{span}

\usetikzlibrary{shapes.geometric, arrows.meta, decorations.markings}

\newcommand\note[1]{\textbf{#1}}

\renewcommand{\epsilon}{\varepsilon}
\newcommand{\R}{\mathbb{R}}

\definecolor{alexmcolor}{RGB}{9,6,250}
\definecolor{amandacolor}{RGB}{50,150,50}

\DeclareMathOperator{\gdiam}{gdiam}
\DeclareMathOperator{\perim}{perim}

\def\R{\mathbb{R}}

\def\E{\mathbb{E}}
\def\vspan{\text{span}}

\date{\today}



\title{Random Shadows of Fixed Polytopes }

\author[Black]{Alexander E. Black\textsuperscript{*}}
\address{\textsuperscript{*}Department of Mathematics, University of California, Davis, CA 95616}

\author[Criado]{Francisco Criado\textsuperscript{\dag}}
\address{\textsuperscript{\dag}Dpto. de Métodos Cuantitativos, CUNEF Universidad, Madrid, Spain 28040}

\begin{document}

\begin{abstract}
Estimating the number of vertices of a two dimensional projection, called a shadow, of a polytope is a fundamental tool for understanding the performance of the shadow simplex method for linear programming among other applications. We prove multiple upper bounds on the expected number of vertices of a random shadow of a fixed polytope. Our bounds are in terms of various parameters in the literature including geometric diameter and edge lengths, minimal and maximal slack, maximal coordinates for lattice polytopes, and maximum absolute values of subdeterminants. For the case of geometric diameter and edge lengths, we prove lower bounds and argue that our upper and lower bounds are both tight for zonotopes. 
\end{abstract}

\maketitle

\section{Introduction}

A shadow of a polytope is its image under a two dimensional linear projection. Shadows appear in various applications including linear programming \cite{borgwardt}, signal processing \cite{DonohoShadows}, computational geometry \cite{CutComplexity, DefProds, SilhouetteBound}, and algebraic complexity theory \cite{NewtonShadows}. In each application, a key parameter is the number of vertices of the shadow. Bounding the number of vertices of a shadow of a polytope is also dual to extension complexity, the minimal number of facets of a polytope that projects onto a given polytope, of polygons \cite{RandExtensions, PolygonExtensions, ShitovPolygonExtensions}. 

In linear programming, the number of vertices of a shadow of a polytope is an upper bound on the number of steps taken by the simplex method using the shadow vertex pivot rule introduced by Gass and Saaty in \cite{GassSaaty}. Namely, to solve a linear program $\max(\{\mathbf{c}^{\intercal}\mathbf{x}: \mathbf{x} \in P\})$ for a polytope $P$, the simplex method generates an initial vertex of the polytope and follows a path on the polytope from vertex to vertex to along edges to reach the optimum. The shadow vertex pivot rule is a method of choosing the path, and the number of vertices of a shadow is an upper bound on the length. 

To better understand the shadow simplex method's performance, Borgwardt showed in \cite{borgwardt, BorgwardtErratum} that the expected number of vertices of a shadow of a random $d$-dimensional polytope with $m$ facets is bounded by a polynomial in $m$ and $d$. In other words, the shadow simplex method runs in polynomial time on random instances, and there are similar follow up results of that type \cite{AdlerSimplex, AsymptoticDiameter}. Spielman and Teng showed that the shadow simplex method takes an expected polynomial number of steps on perturbations of a fixed linear program \cite{SpielmanTeng}. Namely, for a linear program $\max(\{\mathbf{c}^{\intercal} \mathbf{x}: A \mathbf{x} \leq \mathbf{b}\})$, one can add Gaussian noise to each entry of $A$ and $\mathbf{b}$ and bound the expected number of vertices of a random shadow of a randomly perturbed polytope. Recent work has continued to improve on the bounds on the expected sizes of shadows in that setting \cite{SmoothUpperLower, smoothedanalysis, vershynin, ImprovedSmoothedAnalysis}. Kelner and Spielman proved a similar polynomial bound for when only $\mathbf{b}$ is perturbed \cite{KelnerSpielman}. 

In contrast, here we study the case of a fixed polytope and a random shadow. The natural parameters to first try for bounding the expected of number of vertices of a random shadow are the number of facets and dimension. However, an argument of Welzl communicated in \cite{Gartner} shows that there exist polytopes for which a random shadow is exponential in these parameters. Thus, in this work, we will bound the expected number of vertices of random shadow using various other parameters.

For us, we will use the term \textbf{uniformly random shadow} to indicate a shadow $\pi(x) = (U^{\intercal} x, V^{\intercal} x)$, where $U$ and $V$ are uniformly random unitary vectors satisfying $U^{\intercal} V = 0$.

\begin{theorem}
\label{mainthm:shadowbound}
The expected number of vertices $s$ of a uniformly random shadow of a polytope $P \subseteq \R^{n}$ satisfies
\[ \frac{2}{M} \gdiam(P)  \leq s \leq \frac{O(\sqrt{n})\pi}{2m}\gdiam(P),\]
where $\gdiam(P)$ is the geometric diameter of $P$, and $m$ and $M$ are the minimal and maximal lengths of edges of $P$ respectively. 
\end{theorem}

Our approach is to adapt the proof technique of Kelner and Spielman to the setting of a fixed polytope. Also, inspired by \cite{PrimZono}, we show that our bounds are asymptotically tight for zonotopes. 

\begin{theorem}
\label{mainthm:tightness}
For each $n \in \mathbb{N}$, there exist zonotopes $Z_{n}$ and $Z'_{n}$ such that the expected numbers of vertices of random shadows $s_{n}$ and $s_{n}'$ of $Z_{n}$ and $Z_{n}'$ respectively satisfy
\begin{align*}
    s_{n} &= \Theta\left(\frac{gdiam(P)}{M}\right) & s_{n}' = \Theta\left(\sqrt{n} \left(\frac{gdiam(P)}{m}\right)\right), 
\end{align*}
where $m$ is the minimal length of an edge of $Z_{n}'$, $M$ is the maximal length of an edge of $Z_{n}$, and $\text{gdiam}$ denotes the geometric diameter.
\end{theorem}

In \cite{KitaharaMizuno}, they introduced parameters $\gamma$ and $\delta_{KM}$ that measure the maximum and minimum nonzero value respectively of a coordinate for a polytope in standard form $P = \{\mathbf{x}: A \mathbf{x} = \mathbf{b}, \mathbf{x} \geq \mathbf{0}\}$. Note that in their paper, they use the notation $\delta$ instead of $\delta_{KM}$. We adjust the notation to avoid confusion later with a different parameter $\delta$. Then they argued that the simplex method with Dantzig's pivot rule takes at most $n \lceil m \frac{\gamma}{\delta_{KM}} \log\left(m \frac{\gamma}{\delta_{KM}}\right)\rceil$ steps to solve a linear program with $m$ equations and $n$ variables. Observe that $\delta_{KM}$ is a lower bound for the minimal length of an edge, since each edge of a polytope in standard form must increase a coordinate from $0$ to a positive value by the theory of the simplex method as found in \cite{UnderstandingandUsingLP}. Similarly, the polytope $P \subseteq [0, \gamma]^{n}$ by definition, so its geometric diameter is at most $\sqrt{n}\gamma$. These bounds yield a corollary to Theorem \ref{mainthm:shadowbound}.

\begin{cor}
\label{cor:KMtypebound}
Let $P = \{\mathbf{x}: A \mathbf{x} = \mathbf{b}, \mathbf{x} \geq \mathbf{0}\}$, and let $\gamma$ and $\delta_{KM}$ denote the minimum and maximum values of a nonzero coordinate of any vertex of $P$. Then the expected number of vertices of a uniformly random shadow of $P$ is at most $\frac{\gamma}{\delta_{KM}} O(n)$.
\end{cor}

Another motivating application of these bounds is to $(0,k)$-lattice polytopes, polytopes whose vertices are in $[0,k]^{n} \cap \mathbb{Z}^{n}$. Recall that the combinatorial diameter of a polytope is the maximal length of a shortest path between pair of vertices in the graph of the polytope. For example, the combinatorial diameter of $[0,1]^{n}$ is $n$.  For $k =1$, Naddef showed in \cite{naddef} that the combinatorial diameter of a $(0,1)$-lattice polytope is always at most $n$, and in \cite{01Paper}, they showed that this bound is achieved by a shadow path. Kleinschmidt and Onn showed in \cite{origlattpolydiam} that the combinatorial diameter of a $d$-dimensional $(0,k)$-lattice polytope is at most $dk$ with improvements given in \cite{implattdiam, dplattdiam, ComputationalLatticePolyDiam, DezaDezaPaper} with the same asymptotic upper bound $\Theta(dk)$. Lower bounds were proven in \cite{PrimZono, LattZonoDiam, PrimPointPack} with more complicated asymptotics. 

There are approaches to follow polynomial length paths on $(0,k)$-lattice polytopes to solve linear programs \cite{SODAPaper, ShortSimpPaths}, but it remains open whether there is a pivot rule that guarantees the simplex method will follow a path of length polynomial in $d$ and $k$ on lattice polytopes. The shadow simplex method remains a candidate solution to this problem. However, the shadow must be chosen carefully, since lattice polytopes can have exponentially large shadows \cite{NewtonShadows, 01ProblemOrigin}. The geometric diameter of a $(0,k)$-lattice polytope in $\mathbb{R}^{n}$ is at most $\sqrt{n} k$ by definition and the length of a shortest edge is at least $1$, which gives rise to another corollary to Theorem \ref{mainthm:shadowbound}.

\begin{cor}
\label{maincor:latticepolytopes}
Let $P$ be a lattice polytope such that $P \subseteq [0,k]^{n}$. Then the expected number of vertices of a random shadow of $P$ is at most $O(nk)$. 
\end{cor}

Note that the upper bounds are asymptotically the same as the best known upper bounds on combinatorial diameters of $(0,k)$-lattice polytopes. Shadows of lattice polytopes are also of interest in algebraic complexity theory \cite{NewtonShadows}. In that context, they want to bound the largest size of a shadow of a Newton polytope, the convex hull of exponents of a polynomial. For example, a case they highlight of being of interest is the Birkhoff polytope, the convex hull of the permutation matrices and Newton polytope of both the determinant and permanent. In their context, they are interested in finding Newton polytopes with exponentially large shadows, which corresponds to finding lattice polytopes with shadows exponential in $d$ and $k$. Our argument here shows that choosing a shadow uniformly at random is not a good technique to generate large shadows for any Newton polytope. 

A drawback of the bounds shown thus far is that they rely on primal data about the polytope, while the size of a shadow is determined by the polytope's normal fan. Results from this perspective commonly rely on two parameters. The first is the maximum absolute value of a subdeterminant $\Delta$ of the constraint matrix $A$ assuming $A \in \mathbb{Z}^{m \times n}$. The second parameter is $\delta$, the minimal distance from a unit norm ray of a normal cone to the hyper plane spanned by $(n-1)$ other rays. In \cite{Bonifas}, they showed $\delta \geq \frac{1}{n \Delta^{2}}$, and $1/\delta$ is generally used as a proxy for $\Delta$ via that bound.   

The parameter $\delta$ was introduced in the context of combinatorial diameters of polytopes by Bonifas, Di Summa, Eisenbrand, H\"{a}hnle, and Niemeier \cite{Bonifas} in which they proved a bound of at most $O\left(\frac{n^{2.5}}{\delta} \log(n/\delta)\right)$. Follow up works improved on this bound \cite{SpectralApproach} or made it constructive \cite{GeometricRandEdge}.  Brunsch and R\"{o}glin made a weaker version of this bound constructive \cite{BrunschRoglin} by arguing a random shadow path obtained by sampling $\mathbf{c}$ and $\mathbf{w}$ from two normal cones is of vertices of $P$ is of length $O(\frac{mn^{2}}{\delta^{2}})$, where $m$ is the number of equations and $n$ is the number of variables defining the polytope. Together with Gro{\ss}wendt, they later made these bounds algorithmic \cite{BrunschRoglinPlus}. Finally, in \cite{DiscCurv}, Dadush and H\"{a}hnle offered an approach using an exponential distribution and gluing three shadows together to yield a combinatorial diameter bound of at most $O\left(\frac{n^{2}}{\delta} \log\left(\frac{n}{\delta} \right) \right)$. Note that none of their bounds are for uniformly random shadows. We prove in that case, one may find a bound quadratic in $n$ and $1/\delta$. 

\begin{theorem}
\label{thm:conebound}
    Let $P = \{\mathbf{x}: A \mathbf{x} \leq \mathbf{b}\} \subseteq \mathbb{R}^{n}$ be a full dimensional polytope. Suppose furthermore that $\mathbf{b}$ is sufficiently generic so that $P$ is simple. Let $\delta$ be the minimal distance of a unit length ray of any normal cone to the hyper-plane spanned by all other rays of that cone. Then the expected number of vertices of a random shadow of $P$ is bounded by $O\left(\frac{n^{2}}{\delta} \right)$.
\end{theorem} 

We leave it open whether our bound is tight, but we suspect our bounds are best possible in terms of $\delta$ and $n$. Using $\delta \geq \frac{1}{n\Delta^{2}}$, the following is an immediate consequence:

\begin{cor}
Let $P = \{\mathbf{x}: A\mathbf{x} \leq \mathbf{b}\}$ with $A \in \mathbb{Z}^{m \times n}$, and suppose that the maximum absolute value of any subdeterminant of $A$ is bounded by $\Delta$ and $P$ is full dimensional. Then the expected number of vertices of a random shadow of $P$ is at most $O(n^{3} \Delta^{2})$. In particular, if $A$ is totally unimodular, the expected number of vertices is at most $O(n^{3})$. 
\end{cor}

In Section \ref{sec:primal}, we will prove the upper and lower bounds in terms of geometric diameters and several corollaries. Then in Section \ref{sec:dual}, we prove bounds in terms of subdeterminants. One may intuitively think of the bounds in Section \ref{sec:primal} as being polar to the bounds in Section \ref{sec:dual}.

\section{Average Shadows for Bounded Geometric Diameter and Edge Lengths}
\label{sec:primal}

In this section, we follow the convention of using uppercase letters for random variables and lowercase letters for the values of those variables.

To prove our bounds, we estimate the perimeter of the shadow using the geometric diameter and the length of an edge in the shadow. Then the number of edges in the shadow is at least the ratio between the perimeter and the expected length of a shortest edge. To do this computation, we require a probabilistic lemma.

\begin{lem}\label{lem:independence2}

Let $e$ be an edge of a polyhedron $P \subseteq \mathbb{R}^n$. Let $\mathcal{C}$ be its normal cone. Suppose that we choose a random $2$-frame generated by the two vectors $U,V$, which are sampled with a normal distribution. 

Then the following two random variables are independent:

\begin{itemize}
 \item The length $L$ of the image of the projection of $\vec{e}$ into $\text{span}(U,V)$ 
 \item The indicator variable $X$ that is $1$ if and only if $\mathcal{C}$ intersects $\text{span}(U,V)$ and $0$ otherwise.
\end{itemize}
\end{lem}

\begin{proof}
With probability $1$, any two vectors in $\mathbb{R}^{n}$ are both independent and not orthogonal to $e$, so we will assume these two conditions for the remainder of the proof. Note that the result holds if and only if it holds for some isometry and scaling of $P$, so we may assume without loss of generality that $\vec{e}=e_1=(1,0,\dots,0)$.

Let $T=\langle  e_1,U \rangle U+\langle e_1,V\rangle V$ be the projection of $e_1$ into $\vspan (U,V)$, let $T_1$ be $\langle T,e_1\rangle $ and let $T_2$ be the class of unitary generators of $e_1^{\perp}\cap \vspan(U,V)$ in $\mathbb{P}^{n-2}$. The claim is that $L$ depends only on $T_1$, while $X$ depends only on $T_2$, and $T_1$ and $T_2$ are independent as random variables.

The  length of the projection of $e$ into $\vspan(U,V)$ can be  computed from $T_1$ alone: $$T_1= \langle  e_1, T\rangle  = \langle e_1, \langle e_1,U\rangle U + \langle e_1,V\rangle V\rangle  = \langle e_1, U\rangle \langle e_1,U\rangle  + \langle e_1,V\rangle \langle e_1,V\rangle =$$ 
$$  = \langle \langle e_1,U\rangle U+\langle e_1,V\rangle V, \langle e_1,U\rangle U+\langle e_1,V\rangle V\rangle =\langle T,T\rangle =||T||^2 = L^2 .$$
Hence, $L = \sqrt{T_{1}}$ and is therefore determined by $T_{1}$. We can check if $\vspan(U,V)$ intersects $\mathcal{C}$ just by looking at $T_2$ because $T_2$ is the generator of the intersection of $\vspan(U,V)$ with the linear subspace containing $\mathcal{C}$. This determines $X$.

Therefore, it remains to show that $T_1$ and $T_2$ are independent random variables. We observe this is the case because the joint probability distribution $f_{T_1,T_2}$ must be rotationally invariant under isometries of $\mathbb{R}^n$ preserving $e_1$. This is because both the sampling of $\vspan(U,V)$ and the vector $e_1$ are invariant under these isometries.

In other words, for every $t_2,t_2' \in \mathbb{P}^{n-2}$, we have $f_{T_1,T_2}(t_1,t_2)=f_{T_1,T_2}(t_1,t'_2)$.  This implies that $f_{T_2}(t_2)$ is constant, and, in particular:

$$f_{T_1}(t_1) f_{T_2}(t_2)= f_{T_2}(t_2)\int_{t'_2\in \mathbb{P}^{n-2}} f_{T_1,T_2}(t_1,t'_2) d t'_2 = \int_{t'_2\in \mathbb{P}^{n-2}} f_{T_1,T_2}(t_1,t'_2)f_{T_2}(t'_2)d t'_2=$$
$$= f_{T_1,T_2}(t_1,t_2)\int_{t'_2\in\mathbb{P}^{n-2}}f_{T_2}(t'_2)d t'_2=f_{T_1,T_2}(t_1,t_2),$$ for any $t_1\in (0,1),\ \ t_2 \in \mathbb{P}^{n-2}$.

\end{proof}

From a geometric standpoint, one way to intuitively see why the previous lemma holds is that whether an edge is contained in a given shadow is determined exactly by the normal fan of the polytope. At the same time, two polytopes may have the exact same normal fans but have very different distributions of edge lengths. 

 The previous observation allows us to compute the expected perimeter using a random variable for edge length and a random variable for inclusion of an edge in the shadow. The next ingredient allows us to estimate the expected edge length.

\begin{lem} \label{lem:exp_length}
Given an edge $\text{conv}(\mathbf{a}, \mathbf{b})$ of length $\ell$ in $\R^{n}$, the expected length of $\text{conv}(\pi(\mathbf{a}), \pi(\mathbf{b}))$ is $C_n \ell$ where $C_n=\Theta(1/\sqrt{n})$ for $\pi(x) = (U^{\intercal}x, V^{\intercal}x)$, $U, V \in S^{n-1}$, $\langle U,V \rangle = 0$, and $U, V$ sampled uniformly. 

\end{lem}

\begin{proof}

Since $U,V \in S^{n-1}$ and $\langle U, V \rangle = 0$ and are randomly chosen from the sphere, we may precompose with an isometry taking $U$ and $V$ to $e_{1}$ and $e_{2}$. Without loss of generality we may also assume that $\mathbf{a} = 0$ by translating and that $\mathbf{b} \in S^{n-1}$ by rescaling, since $\pi$ is linear. Then $\mathbf{b}$ may instead be taken as a uniform random vector $B\sim U(\mathbb{S}^{n-1})$ and $\pi$ may be fixed to be $e_{1}^{\intercal} x, e_{2}^{\intercal}x$. Hence, $\pi(x) = (x_{1}, x_{2})$. Then $C_n=\E[||(B_1,B_2)||_2]$.

We begin by computing the exact value of $\E[|B_1|]$. Figure 1 \ref{fig:SphereSlices} illustrates how to compute the probability density function of $B_1$:
\begin{figure}
\label{fig:SphereSlices}
    \centering
    \begin{tikzpicture}[scale=1, every node/.style={scale=1}]
  \shade[ball color = gray!40, opacity = 0.4] (0,0) circle (2cm);
  \draw (0,0) circle (2cm);
  \draw (-2,0) arc (180:360:2 and 0.6);
  \draw[dashed] (2,0) arc (0:180:2 and 0.6);
  \draw (-1.8,0.8) arc (180:360:1.8 and 0.5);
  \draw[dashed] (1.8,0.8) arc (0:180:1.8 and 0.5);
  \draw (-1.414,1.3) arc (180:360:1.414 and 0.42);
  \draw[dashed] (1.414,1.3) arc (0:180:1.414 and 0.42);
  \draw[->] (0,-2.5) -- (0,2.5) node[above] {\(x_1\)};
  \draw[->] (-2.5,0) -- (2.5,0) node[right] {\(x_{\text{others}}\)};
  \node at (4,2) {\(ds = \frac{1}{\sqrt{1-t^2}}\)};
  
  \draw[<->, thick] (0,-.1) -- (2,-.1) node[midway,below] {1};
  \draw[<->, thick] (-2.1,0) -- (-2.1,0.8) node[midway,left] {\(t\)};
  \draw[<->, thick] (-2.1,.8) -- (-2.1,1.3) node[midway,left] {\(dt\)};
  \draw[<->, thick] (0,2.1) -- (1.414,2.1) node[midway,above] {\(r=\sqrt{1-t^2}\)};
  \node at (4,1) {\(A_{\text{slice}} = A_{n-2}(r)ds\)};

\end{tikzpicture}

        \caption{A visualization of the proof of Lemma \ref{lem:exp_length}.}
\end{figure}
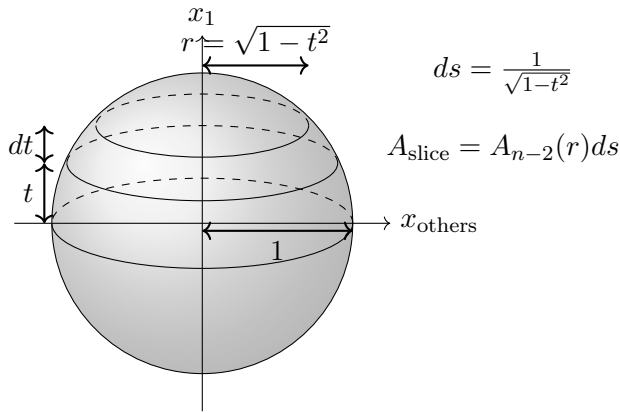

$$ f_{B_1}(t) = \int_{x\in \sqrt{1-t^2} \mathbb{S}^{n-2}} \frac{1}{A_{n-1}(1)} ds,$$

where $A_{n-1}(r)$ is the area of the $n-1$-sphere of radius $r$ and $ds$ is the differential of area:

\[ds=\left \|\frac{d(t, \sqrt{1-t^2})}{dt}\right \|_2= \left \|\left(1,-\frac{t}{\sqrt{1-t^2}}\right)\right\|_2=  \sqrt{1 + \frac{t^{2}}{1-t^{2}} } = \sqrt{\frac{1}{1-t^2}}.\]
Then

\[ f_{B_1}(t) = \frac{A_{n-2}(\sqrt{1-t^2})}{(\sqrt{1-t^2})A_{n-1}(1)}. \]

Recall that $A_n(r)=\frac{2\pi^{(n+1)/2}}{\Gamma((n+1)/2)}r^n$. Then, by simplifying terms, we find that

\[ f_{B_1}(t) = \dfrac{\frac{2 \pi^{(n-1)/2}}{\Gamma((n-1)/2)} (\sqrt{1-t^{2}})^{n-2}}{\sqrt{1-t^{2}} \frac{2 \pi^{n/2}}{\Gamma(n/2)}} = \frac{\sqrt{\pi} \Gamma(n/2)}{\Gamma((n-1)/2)} (1-t^2)^{(n-3)/2}.\]

We integrate to compute the expectation of $|B_1|$:
\begin{align*}
    E[|B_1|] &= \int_{-1}^{1} |t| f_{B_1}(t) dt = \int_0^1 2tf_{B_1}(t) dt = \frac{2\sqrt{\pi} \Gamma(n/2)}{\Gamma((n-1)/2)} \int_0^1 t(1-t^2)^{(n-3)/2}dt \\
    &= -\frac{\sqrt{\pi} \Gamma(n/2)}{\Gamma((n-1)/2)} \int_{1}^{0} u^{(n-3)/2} du = \frac{\sqrt{\pi} \Gamma(n/2)}{\Gamma((n-1)/2)} \left(\frac{2}{n-1}\right). 
\end{align*}

 In Theorem $2$ of \cite{GammaFunctionBound}, Chu showed that 
\[\sqrt{\frac{(2n-1)(n-1)}{(2n-2)2}} \leq \frac{\Gamma(n/2)}{\Gamma((n-1)/2)} \leq \sqrt{\frac{(2n-2)(n-1)}{(2n-3)2}}.\]

It follows that that $\frac{\Gamma(n/2)}{\Gamma((n-1)/2)} = \Theta(\sqrt{n})$, so $\E[|B_{1}|] = \Theta(1/\sqrt{n})$. By linearity of expectation $\E[\|(B_{1},B_{2})\|_{1}] = \Theta(1/\sqrt{n})$. Since the $2$-norm and $1$-norm differ by a constant factor in dimension $2$, $\E[\|(B_{1},B_{2})\|_{2}] = \Theta(1/\sqrt{n})$.

\end{proof}

It remains to put these ingredients together to establish a general bound on the expected number of edges in a shadow of a given polytope.

\begin{proof}[Proof of Theorem \ref{mainthm:shadowbound}]

Let $P$ be a polytope. Then we can estimate the size of a shadow in the following way. Let $X(U,V)$ be the random variable for a shadow of $P$ with respect to the projection $\pi(x) = (U^{\intercal}x, V^{\intercal} x)$. Then we have, by Lemma \ref{lem:independence2}, that  
\[\E[\perim(\pi(P))] = \sum_{e \in E(P)} \E[||\pi(e)||_{2}] Pr( \pi(e) \in E(\pi(P))) = \sum_{e\in E(P)} C_n ||e||_{2} Pr(\pi(e)\in E(\pi(P))), \]
where $E(\pi(P))$ is the set of edges of $\pi(P)$. Observe that, by linearity of expectation, $\E[|E(\pi(P))|]= \sum_{e\in E(P)} Pr(\pi(e)\in E(\pi(P)))$. Let $m$ denote the length of the shortest edge, and let $M$ denote the length of the largest edge. Then we have the following:
\begin{align*}
    \E[\perim(\pi(P))] &= \sum_{e\in E(P)} C_n ||e||_2 Pr\left(\pi(e)\in E(\pi(P))\right) \leq C_n M \sum_{e\in E(P)} Pr\left(\pi(e)\in E(\pi(P))\right) \\
    &= C_n M \E[|E(\pi(P))|],    \\
    \E[\perim(\pi(P))] &= \sum_{e\in E(P)} C_n ||e||_2 Pr\left(\pi(e)\in E(\pi(P))\right) \geq C_n m \sum_{e\in E(P)} Pr\left(\pi(e)\in E(\pi(P))\right) \\ 
    &= C_n m \E[|E(\pi(P))|].
\end{align*}

Then we have bounds on the expected size of the shadow in terms of the expected perimeter of the shadow:

\[ \label{eq:proofstep1} \frac{\E[\perim(\pi(P))]}{C_n m} \leq \E[|E(\pi(P))|] \leq \frac{\E[\perim(\pi(P))]}{C_n M}.\]

Now we relate the expected perimeter perimeter to the expected geometric diameter of the projection. By the triangle inequality and Barbier's theorem,
\[2\gdiam(\pi(P)) \leq \perim(\pi(P)) \leq \pi \gdiam(\pi(P)),\]
and, taking expected values, 
\[2\E[\gdiam(\pi(P))] \leq \E[\perim(\pi(P))] \leq \pi \E[\gdiam(\pi(P))],\]

We apply it to Equation~\ref{eq:proofstep1}:

\[ \label{eq:proofstep2} \frac{2\E[\text{gdiam}(\pi(P))]}{C_n m} \leq \E[|E(\pi(P))|] \leq \frac{\pi\E[\text{gdiam}(\pi(P))]}{C_n M}.\]

Finally we relate the expected geometric diameter of the shadow with the geometric diameter of $P$. For any projection, the length of the projection of the diameter is less than the diameter of the shadow of $P$. By integration across all projections, $C_n \gdiam(P) \leq \E[\gdiam(\pi(P))]$.

On the other hand, for any projection, the geometric diameter of the projection must not be larger than the geometric diameter of $P$. By integration across all projections, $\E[\gdiam(\pi(P))] \leq \gdiam(P)$.

We apply these inequalities to Equation~\ref{eq:proofstep2}:

\[\frac{2 \gdiam(P)}{M}\leq \E[|E(\pi(P))|] \leq  \frac{\pi \gdiam(P))}{C_n m}.\]
By Lemma \ref{lem:exp_length}, we must have that
\[ \frac{2}{M} \gdiam(P)  \leq \E[|E(\pi(P))|] \leq \frac{\Theta(\sqrt{n})\pi}{2m}\gdiam(P) .\]
\end{proof}

To apply this to the case of the Birkhoff polytope, we require the following easy lemma that follows from its edge structure:

\begin{lem}
For the Birkhoff polytope $B_{n}$, we have the following:
\begin{itemize}
    \item[(i)] The geometric diameter of $B_{n}$ is $\sqrt{n}$ 
    \item[(ii)] The shortest edge of $B_{n}$ has length $\sqrt{2}$
    \item[(iii)] The longest edge of $B_{n}$ has length $\sqrt{n}$.
    \item[(iv)] The dimension of $B_{n}$ is $n^{2}-2n+1$.
\end{itemize}
\end{lem}

Then we have the following bound on the expected size of a shadow:
\begin{cor}
The expected size $\E[s(B_{n})]$ of a random shadow of $B_{n}$ satisfies 
\[4 \leq \E[s(B_{n})] \leq \Theta(n^{3/2}). \]
\end{cor}

We have similar corollaries for other classes of polytopes. Namely, in \cite{NewtonShadows}, they considered Newton polytopes.  

\begin{lem} \label{lem:newton}
Let $P$ be a Newton Polytope for a polynomial $p \in \R[x_{1}, x_{2}, \dots, x_{n}]$ of degree $k$. Then we have 
\begin{itemize}
    \item[(i)] The geometric diameter of $P$ is at most $\sqrt{n}k$
    \item[(ii)] The shortest edge of $P$ is of length at least $1$
    \item[(iii)] The longest edge of $P$ has length at most $\sqrt{n}k$.
\end{itemize}
\end{lem}

\begin{proof}
Since $p$ is of degree $k$, $P \subseteq [0,k]^{n}$. Furthermore, by definition of a Newton polytope, $P$ is a lattice polytope. The distance between two lattice points in the $[0,k]^{n}$-cube is at most $\sqrt{n}k$, which proves $(i)$ and $(iii)$. The distance between any two lattice points is at least $1$, which proves (ii).
\end{proof}

As in the case of the Birkhoff polytope, this lemma tells us that random shadows of Newton polytopes tend to be very small: 

\begin{cor}
Let $P$ be a Newton Polytope for a polynomial $p \in \R[x_{1}, x_{2}, \dots, x_{n}]$ of degree $k$. Then the expected size $\E[s(P)]$ of a random shadow of $P$ satisfies 
\[ 1 \leq \E[s(P)] \leq \Theta(nk).\]
\end{cor}

These results also have implications for more general rational polytopes with vertices with bounded numerator and denominator. Such a condition appears in the linear programming literature. Namely, consider a polytope of the form $P =\{\mathbf{x}: A \mathbf{x} = \mathbf{1}, \mathbf{x} \geq \mathbf{0}\}$ for $A$ integral. Then, by Proposition $3.1$ and Theorem $3.8$ of \cite{CircuitImbalances}, there is a bound on the numerator and denominator of each vertex of a polytope in terms of circuit imbalance measures. Essentially the same proof as for Lemma \ref{lem:newton} applies to that case:

\begin{lem} \label{lem:rational}
Let $P \subseteq \mathbb{R}^{n}$ be a polytope, and suppose that $P$ is rational and for each $\mathbf{v} \in V(P)$, $\mathbf{v}_{i} = \frac{p}{q}$ for some $p \leq \alpha$ and $q \leq \beta$. Then we have 
\begin{itemize}
    \item[(i)] The geometric diameter of $P$ is at most $\sqrt{n}\alpha$
    \item[(ii)] The shortest edge of $P$ is of length at least $\frac{1}{\beta^{2}}$
    \item[(iii)] The longest edge of $P$ has length at most $\sqrt{n}\alpha$.
\end{itemize}
\end{lem}

Note that these rational polytopes are more than just rescalings of Newton polytopes. In order to clear denominators simultaneously one would need to multiply by the least common multiple of all denominators, which can be of exponential size.  For these polytopes, we still find nice bounds:

\begin{cor}
\label{cor:rationalpoly}
Let $P \subseteq \mathbb{R}^{n}$ be a polytope, and suppose that $P$ is rational and for each $\mathbf{v} \in V(P)$, $\mathbf{v}_{i} = \frac{p}{q}$ for some $p \leq \alpha$ and $q \leq \beta$. Then the expected size of a random shadow of $P$ is at most $\Theta(n\beta^{2} \alpha).$
\end{cor}

Then Corollary \ref{maincor:latticepolytopes} follows immediately as a consequence of Corollary \ref{cor:rationalpoly}.

\subsection{Tightness for Zonotopes}

For lower bounds, we are interested in computing examples of polytopes $P$ for which $\frac{2}{M} \text{gdiam}(P)$ is an approximately tight bound for the expected size of a shadow. To do this, one may take $k$ parallel segments all of length $1$ and mildly perturb each of them. Then take the zonotope $Z$ given by the Minkowski sum of those segments. The longest edge will have length approximately $1$, and the geometric diameter will be approximately $k$. For a zonotope every shadow generically has the same the number of edges given by twice the number of nonparallel segments in its Minkowski sum decomposition. Hence, the expected number of vertices of a random shadow is approximately $2k = \frac{2}{M} \text{gdiam}(Z)$.

To prove tightness of our upper bounds, we apply a similar construction. For the hyper-cube $[0,1]^{n}$, the expected size of a random shadow is $2n$, the edge length is $1$, and the geometric diameter is $\sqrt{n}$. Hence, for the hyper-cube the upper bound is asymptotically tight. Thus, Theorem \ref{mainthm:tightness} follows from these examples. We can strengthen our upper bounds to allow for arbitrarily many generators to show the cube is not so special. Namely, one may add mild perturbations to each of the standard basis vectors to add $n$ new generators but maintain the tightness of the bound. By using this observation and the result that the expected size of a random shadow of a zonotope is always twice the number of generators, we find a geometric corollary to our results.

\begin{cor}
\label{cor:convexgeoresult}
Given a zonotope with $N$ generators in $\mathbb{R}^{n}$ all of unit length, the minimal geometric diameter is $\Theta(N/\sqrt{n})$
\end{cor}

Corollary \ref{cor:convexgeoresult} appears in different terminology as Theorem $4$ of \cite{PolarizationPaper} in the context of convex geometry. Though their techniques to prove this result are very different from ours.

\section{Average Shadows for Bounded Subdeterminants}
\label{sec:dual}

To prove Theorem \ref{thm:conebound}, we will instead bound the number of edges of a polar analog of the shadow, which has the same number of edges as the shadow. The strategy to bound the number of edges is the same as in the previous sections but instead using the dual picture. We bound the perimeter and bound the expected length of an edge. The ratio then yields a bound on the number of vertices of the shadow. We will bound the expected length of an edge after the proof using a technical observation given by Corollary \ref{cor:ExpEdgeLength}, but first we give the proof assuming the bound. 

\begin{proof}[Proof of Theorem \ref{thm:conebound}]

We consider the normal fan of $P$. Consider a random plane spanned by vectors $\mathbf{c}$ and $\mathbf{w}$ and a unit circle in the plane spanned by $\mathbf{c}$ and $\mathbf{w}$. This unit circle is split into components by the normal fan $\mathcal{N}$ of $P$, and the total number of components is the length of the shadow path. In particular, for a generic choice of plane the unit circle intersects finitely many codimension $1$ cones of the normal fan at points given in order cyclically as $\mathbf{x}_{1}, \mathbf{x}_{2}, \dots, \mathbf{x}_{\ell}$. Consider the polygon $S$ given by taking the convex hull of these points. The number of edges of this polygon is precisely the number of vertices in the shadow. Thus, we will bound the number of edges in this polygon instead of the shadow polygon. Again our argument relies on edge lengths. 

Note that since each $\mathbf{x}_{i}$ lies on the unit circle, the perimeter of the resulting polygon is at most $2 \pi$. For each cone $C$, let $X_{C}$ be the event that the normal cone is intersected on its interior by the plane and $Y_{C}$ be the length of an edge of $S$ given it intersects the interior of $C$. Then we have
\[2 \pi \geq \E[\text{Perim}(S)] = \sum_{C \in \mathcal{N}(C)} \E[Y_{C}] Pr[X_{C}] \geq \min_{C}( \E[Y_{C}]) \sum_{C \in \mathcal{N}(C)} Pr[X_{C}] = \min_{C}(\E[Y_{C}]) \E[|V(S)|],\]
where $\text{Perim}(S)$ denotes the perimeter of $S$ and $V(S)$ denotes the set of vertices of $S$. Thus, $\frac{2 \pi}{ \min_{C}(\E[Y_{C}])} \geq\E[|V(S)|]$, so lower bounding $\E[Y_{C}]$ suffices to upper bound $\E[|V(S)|]$. 

Let $Z_{F}$ be the random variable given by the length of an edge in $C$ given it intersects a facet $F$ of $C$. Then $\E[Y_{C}] \geq \min_{F} \E[Z_{F}]$. Therefore, we may condition on a vertex being in a fixed facet $F$ of $C$. Then the distribution of endpoints in $F \cap S^{n-1}$ is uniform. The length of the edge is lower bounded by the minimal distance of the point in $F \cap S^{n-1}$ to a point in the union of all hyper-planes spanned by $F'$ for some $F' \neq F$. Hence, the expected length of an edge is lower bounded by the expected distance of a uniformly random point in $F \cap S^{n-1}$ to the arrangement of hyperplanes spanned by $F'$ for each facet $F' \neq F$. By Corollary \ref{cor:ExpEdgeLength}, this length is at least $\frac{\delta}{8n^{2}}$, which completes the proof. 

\end{proof}

It remains to bound the expected distance for a uniformly random point in a cone intersected with a sphere to a supporting hyper-plane. We do this first for a cone intersected with a ball.

\begin{lem}
\label{lem:ConeDistance}
Let $C = \text{cone}(\mathbf{v}_{1}, \dots, \mathbf{v}_{k}) \subseteq \mathbb{R}^{n}$ be a simplicial cone such that $\mathbf{v}_{1}, \dots, \mathbf{v}_{k}  \in S^{n-1}$ with $k \geq 2$. Let $H \supseteq \text{span}(\mathbf{v}_{2}, \dots, \mathbf{v}_{k})$ be a hyper-plane, and let $h = d(\mathbf{v}_{1}, H)$. Then 
\[Pr[d(\mathbf{x},H) \leq \varepsilon: x \in C \cap B^{n}] \leq \frac{(1+\varepsilon)^{k}-1}{h}. \]
\end{lem}

\begin{proof}
Define $H_{y} = H + y\mathbf{w}$ for $y \in \mathbb{R}$, where $\mathbf{w}$ is the unique normal vector to $H$ such that $\mathbf{v}_{1} - h\mathbf{w} \in H$. In particular, by definition, $\|\mathbf{w}\|_{2} = 1$. Then we have
\begin{equation}
\label{eqn:probcone}
Pr[d(\mathbf{x},H) \leq \varepsilon] = \frac{1}{Vol_{k}(C \cap B^{n})} \int_{0}^{\varepsilon} Vol_{k-1}(H_{y} \cap C \cap B^{n}) dy.    
\end{equation}

Consider the pyramid $X = \text{conv}(\{\mathbf{v}_{1}\} \cup (C \cap H \cap B^{n})\})$. Then $X \subseteq C \cap B^{n}$, so $Vol_{k}(X) \leq Vol_{k}(C \cap B^{n})$. Note that $Vol_{k}(X) = \frac{h}{k} Vol_{k-1}(H \cap C \cap B^{n})$. Hence, we have
\begin{equation}
\label{eqn:pyramidbnd}
\frac{k}{h Vol_{k-1}(H \cap C \cap B^{n})} \geq \frac{1}{Vol_{k}(C \cap B^{n})}.    
\end{equation}

Consider the cone $C' = \text{cone}(\mathbf{v}_{2}, \dots, \mathbf{v}_{k})$. Then we have
\[H_{y} \cap C = H_{y} \cap (y \mathbf{v}_{1} + C').\]
It follows that
\[H_{y} \cap C \cap B^{n} = \{y \mathbf{v}_{1}\} \sqcup \bigsqcup_{\mathbf{z} \in C' \cap S^{n-1}} \{y \mathbf{v}_{1} + \lambda \mathbf{z}: y\mathbf{v}_{1} +\lambda \mathbf{z} \in B^{n}, \lambda > 0\}. \]
For $\mathbf{z} \in C' \cap S^{n-1}$, by the reverse triangle inequality, 
\[\|y \mathbf{v}_{1} + \lambda \mathbf{z}\|_{2} \geq \lambda-y. \]
Hence, the maximal such $\lambda$ such that $y \mathbf{v}_{1} + \lambda \mathbf{z} \in B^{n}$ is at most $1+y$. It follows that 
\[H_{y} \cap C \cap B^{n} \subseteq y \mathbf{v}_{1} + (1+y)(C' \cap B^{n}).\]
Hence, we have 
\begin{align*}
    \int_{0}^{\varepsilon} Vol_{k-1}(H_{y} \cap C \cap B^{n}) dy &\leq \int_{0}^{\varepsilon} Vol_{k-1}((1+y) (H \cap C \cap B^{n})) dy \\
    &= \int_{0}^{\varepsilon} (1+y)^{k-1} Vol_{k-1}(H \cap C \cap B^{n})) dy \\ 
    &= \frac{1}{k}\left( (1+\varepsilon)^{k} - 1 \right)  Vol_{k-1}(H \cap C \cap B^{n}).
\end{align*}
By combining this inequality with Equations \ref{eqn:probcone} and \ref{eqn:pyramidbnd},
\begin{align*}
   Pr[d(\mathbf{x},H) \leq \varepsilon] &\leq \frac{k}{h Vol_{k-1}(H \cap C \cap B^{n})} \left( \frac{1}{k}\left( (1+\varepsilon)^{k} - 1 \right)  Vol_{k-1}(H \cap C \cap B^{n})\right) \\ 
   &= \frac{(1+\varepsilon)^{k}-1}{h}.
\end{align*}

\end{proof}

Using a union bound, we translate this estimate into a bound on the expected minimal distance to an arrangement of supporting linear hyper-planes for the cone. 

\begin{lem}
\label{lem:arrangementbound}
Let $\mathcal{H} = \bigcup_{i=1}^{k} H_{i}$ be a linear hyper-plane arrangement. Let $C = \text{cone}(\mathbf{v}_{1}, \dots, \mathbf{v}_{k})$ be a cone, and suppose that $d(\mathbf{v}_{i}, H_{j}) \geq h$ for $i = j$ and is $0$ otherwise for all $i,j \in [k]$. Then 
\[\E[d(\mathbf{x},\mathcal{H}): x \in C \cap B^{n}] \geq \frac{h}{8k^{2}}. \]
\end{lem}

\begin{proof}
By Lemma \ref{lem:ConeDistance}, we have $Pr[d(\mathbf{x},H_{i}) \leq \varepsilon: x \in C \cap B^{n}] \leq \frac{(1+\varepsilon)^{k}-1}{h}$. Note that $d(\mathbf{x},\mathcal{H}) = \min_{i \in [k]} d(\mathbf{x},H_{i})$. It follows that
\[Pr[d(\mathbf{x}, \mathcal{H}) \leq \varepsilon: \mathbf{x} \in C \cap B^{n}] \leq \sum_{i = 1}^{k} Pr([d(\mathbf{x}, \mathcal{H}) \leq \varepsilon: \mathbf{x} \in C \cap B^{n}]) \leq \frac{k((1+\varepsilon)^{k} -1)}{h}.\]
Then, by Markov's inequality, 
\[\E[d(\mathbf{x}, \mathcal{H}): \mathbf{x} \in C \cap B^{n}] \geq \varepsilon\left(1-  \frac{k((1+\varepsilon)^{k} -1)}{h}\right).\]
Note that this equation is true for all choices of $\varepsilon$, so we will choose $\varepsilon$ carefully to yield our desired bound. Recall that $\left(1 + x \right)^{k} \leq e^{kx}$ for $x \geq -1$. Recall also that for $x \leq 1$, $e^{x} \leq 1 + x + x^{2}$. Suppose that $\varepsilon \leq \frac{1}{k}$. Then we may rewrite this bound as follows
\begin{align*}
    \varepsilon \left(1 - \frac{k((1+\varepsilon)^{k}-1)}{h}\right) 
    &= \varepsilon\left(1 + \frac{k}{h} - \frac{k(1+\varepsilon)^{k}}{h}\right) \\
    &\geq \varepsilon\left(1 + \frac{k}{h} - \frac{ke^{k\varepsilon}}{h} \right) \\
    &\geq  \varepsilon\left(1 + \frac{k}{h} - \frac{k + k^{2}\varepsilon + k(k\varepsilon)^{2}}{h} \right) \\
    &= \varepsilon\left(1 - \frac{k^{2}\varepsilon}{h} - \frac{k(k\varepsilon)^{2}}{h} \right).
\end{align*}
 Note that $\frac{h}{4k^{2}} \leq \frac{1}{k}$, since $h \leq 1$. Letting $\varepsilon = \frac{h}{2k^{2}}$ yields
\[\E[d(\mathbf{x}, \mathcal{H}): \mathbf{x} \in C \cap B^{n}] \geq \frac{h}{2k^{2}}\left( 1 - \frac{1}{2} - \frac{h}{4k} \right) 
    \geq \frac{h}{2k^{2}}\left( 1 - \frac{1}{2} - \frac{1}{4} \right)
    \geq \frac{h}{8k^{2}}.\]
    
\end{proof}

However, this lemma is not sufficient for our purposes. We need a bound over the expected distance to a hyper-plane arrangement intersected with the sphere $S^{n-1}$ instead of the unit ball $B^{n}$. We can transfer our bounds from the ball to the sphere by noting the distance to the arrangement is maximized on the sphere.

\begin{lem}
\label{lem:VolumetoSurfaceArea}
Let $B^{n}$ denote the unit ball of radius $1$, and let $C$ be a pointed polyhedral cone with vertex at the origin. Let $f: B^{n} \to \mathbb{R}$ be a function such that $f(\mathbf{x}) \leq f\left(\frac{\mathbf{x}}{\|\mathbf{x}\|_{2}}\right)$ for all $\mathbf{x} \in B^{n}$. Then we have
\[\E[f(\mathbf{x}): \mathbf{x} \in B^{n} \cap C] \leq \E[f(\mathbf{x}): \mathbf{x} \in S^{n-1} \cap C].\]
\end{lem}

\begin{proof}
Let $k$ be the dimension of $C$. Note first that 
\begin{align*}
    Vol_{k}(B^{n} \cap C) &= \int_{0}^{1} r^{k-1} Vol_{k- 1}(S^{n-1} \cap C) dr \\
    &= \frac{Vol_{k-1}(S^{n-1} \cap C)}{k}.
    \end{align*}
It follows that
\begin{align*}
    \E[f(\mathbf{x}): \mathbf{x} \in B^{k} \cap C] &= \frac{1}{Vol_{k}(B^{k} \cap C)} \int_{\mathbf{z} \in B^{n}} f(\mathbf{z}) d \mathbf{z} \\
    &= \frac{1}{Vol_{k}(B^{k} \cap C)} \int_{0}^{1} r^{k-1} \int_{\mathbf{x} \in S^{k-1}\cap C} f(r\mathbf{x}) d\mathbf{x} dr   \\
    &= \frac{k}{Vol_{k-1}(S^{n-1} \cap C)} \int_{0}^{1} r^{k-1} \int_{\mathbf{x} \in S^{n-1} \cap C} f(r\mathbf{x}) d\mathbf{x} dr \\
    &\leq \frac{k}{Vol_{k-1}(S^{k-1} \cap C)} \int_{0}^{1} r^{k-1} \int_{\mathbf{x} \in S^{n-1} \cap C} f\left(\mathbf{x}\right) d\mathbf{x} dy  \\
    &= \frac{1}{Vol_{k-1}(S^{k-1} \cap C)} \int_{\mathbf{x} \in S^{n-1} \cap C} f(\mathbf{x}) d \mathbf{x} \\
    &= \E[f(\mathbf{x}): \mathbf{x} \in S^{n-1})].
\end{align*}
\end{proof}
It remains to apply this observation to our case. 
\begin{cor}
\label{cor:ExpEdgeLength}
Let $\mathcal{H} = \bigcup_{i=1}^{k} H_{i}$ be a linear hyper-plane arrangement. Let $C = \text{cone}(\mathbf{v}_{1}, \dots, \mathbf{v}_{k})$ be a cone, and suppose that $d(\mathbf{v}_{i}, H_{j}) \geq h$ for $i = j$ and is $0$ otherwise for all $i,j \in [k]$. Then 
\[\E[d(\mathbf{x},\mathcal{H}): \mathbf{x} \in C \cap S^{n-1}] \geq \frac{h}{8k^{2}}. \]
\end{cor}

\begin{proof}
Consider $f(\mathbf{x}) = d(\mathbf{x}, \mathcal{H}) \chi_{C}(\mathbf{x})$, where $\chi_{C}$ is the indicator function of $C$, and note this function increases along rays. Then apply Lemmas \ref{lem:arrangementbound} and \ref{lem:VolumetoSurfaceArea} to complete the proof. 
\end{proof}

\subsection{Lower Bounds from Zonotopes}

Recall that the permutahedron $\Pi_{n}$ is a zonotope with generators parallel to $\{e_{i} - e_{j} : 1 \leq i < j \leq n\}$. It is a simple $(n-1)$-dimensional polytope in $\mathbb{R}^{n}$ contained in the hyper-plane with sum of coordinates fixed to $\binom{n+1}{2}$. We consider a full dimensional augmentation $\tilde{\Pi}_{n} = \Pi_{n} + \ell$, where $\ell$ is the segment parallel to the all ones vector. Then $\tilde{\Pi}_{n}$ is full-dimensional and is the prism of $\Pi_{n}$ and therefore simple. 

Note that the expected number of vertices of a random shadow of $\tilde{\Pi}_{n}$ is twice its number of generators, which is $n(n-1) +2$ for $\tilde{\Pi}_{n}$. For $S \subseteq [n]$, let $e_{S} = \sum_{s \in S} e_{s}$. A normal cone for a vertex of $\Pi_{n}$ has, up to symmetry, ray generators $e_{1}, e_{[2]}, \dots, e_{[n-1]}$. For $\Tilde{\Pi}_{n}$, the ray generators of a normal cone are, up to symmetry,
\[e_{1} - \frac{1}{n} e_{[n]}, e_{[2]} - \frac{2}{n} e_{[n]}, \dots, e_{[n-1]} - \frac{n-1}{n} e_{[n]}, e_{[n]},\]
and the facet defining inequalities of the cone are $e_{i+1} - e_{i}$ and $\sum_{i=1}^{n} e_{n}$. The unique inequality not tight at $e_{[k]} - \frac{k}{n} e_{[n]}$ is $e_{k+1} - e_{k}$. To figure out the distance from that ray to the hyper-plane from that inequality, we compute 
\[\frac{|(e_{k+1}- e_{k})^{\intercal} (e_{[k]} - \frac{k}{n} e_{[n]})| }{\|e_{k+1} 
-e_{k}\|_{2} \|e_{[k]} - \frac{k}{n} e_{[n]}\|_{2}} = \frac{\frac{n-k}{n} + \frac{k}{n}}{\sqrt{2} \frac{(n-k)\sqrt{k} + k \sqrt{n-k}}{n}} = \frac{n}{\sqrt{2}((n-k)\sqrt{k} + k \sqrt{n-k})}.\]
The unique inequality not tight at $e_{[n]}$ is $e_{[n]}$, and in that case the distance is $1$. Then we have
\[\delta(\Tilde{\Pi}_{n}) = \min_{k \in [n]}\left(1,  \frac{n}{\sqrt{2}((n-k)\sqrt{k} + k \sqrt{n-k})}\right) = \Theta(1/\sqrt{n}).\]
Then the expected number of vertices of a random shadow of $\Tilde{\Pi}_{n}$ is $\Theta(n^{1.5}/\delta)$, which is off by a factor of $\sqrt{n}$ from the upper bound we show. We leave finding a tight bound as an open problem. 

Note that $\Theta(n^{1.5}/\delta)$ is also a lower bound for the maximum combinatorial diameter in terms of $\delta$ and $n$, since the combinatorial diameter of a zonotope is its number of generators. To our knowledge, there are no matching lower bounds for combinatorial diameters of polytopes to the upper bounds in terms of $n$, $\delta$, and $\Delta$ found in \cite{Bonifas, DiscCurv}, and we leave finding such lower bounds as another interesting open problem.

\section*{Acknowledgments}  
This work started as a result of the 2023 ICERM semester program on discrete optimization. The first author was also supported by the NSF GRFP. We thank Sophie Huiberts for directing us to the work of Kelner and Spielman as well as other useful discussions.

\bibliographystyle{amsplain}
\bibliography{bibliography.bib} 

	
\end{document}